\documentclass{article}
\usepackage[utf8]{inputenc}
\usepackage[left=1in,top=1in,right=1in,bottom=1in]{geometry}
\usepackage[linktocpage=true]{hyperref}
\usepackage{setspace}
\usepackage{tikz}
\usetikzlibrary{decorations.pathreplacing}
\usetikzlibrary{calc}
\usepackage{amssymb, amsmath, amsthm, graphicx,mathrsfs}
\usepackage{caption,cite}
\usepackage{subcaption}
\usepackage{mathtools,color}

\usepackage{cleveref}
\usepackage{todonotes}

\newtheorem{thm}{Theorem}

\newtheorem{rem}{Remark}
\newtheorem{lem}[thm]{Lemma}
\newtheorem{conj}{Conjecture}
\newtheorem{prop}[thm]{Proposition}

\newtheorem{ques}{Question}

\newcommand{\N}{\mathbb{N}}

\newcommand{\ex}{\mathrm{ex}}
\newcommand{\exr}{\mathrm{ex^*}}
\newcommand{\deltar}{\delta^*}

\newcommand{\floor}[1]{\left\lfloor #1 \right\rfloor}

\title{Rainbow Erd\H{o}s-S\'os Conjectures}

\author{Nicholas Crawford\footnote{Department of Mathematical and Statistical Sciences, University of Colorado Denver {\tt nicholas.2.crawford@ucdenver.edu}.}\and Dylan King\footnote{Department of Mathematics, California Institute of Technology {\tt dking@caltech.edu}} \and Sam Spiro\footnote{Department of Mathematics, Rutgers University {\tt sas703@scarletmail.rutgers.edu}. This material is based upon work supported by the National Science Foundation Mathematical Sciences Postdoctoral Research Fellowship under Grant No. DMS-2202730.}}
\date{\today}

\begin{document}
\maketitle
\begin{abstract}
    An edge colored graph is said to contain rainbow-$F$ if $F$ is a subgraph and every edge receives a different color. In 2007, Keevash, Mubayi, Sudakov, and Verstra\"ete introduced the \emph{rainbow extremal number} $\exr(n,F)$, a variant on the classical Tur\'an problem, asking for the maximum number of edges in a $n$-vertex properly edge-colored graph which does not contain a rainbow-$F$. In the following years many authors have studied the asymptotic behavior of $\exr(n,F)$ when $F$ is bipartite. In the particular case that $F$ is a tree $T$, the infamous Erd\"os-S\'os conjecture says that the extremal number of $T$ depends only on the size of $T$ and not its structure. After observing that such a pattern cannot hold for $\exr$ in the usual setting, we propose that the relative rainbow extremal number $\exr(Q_n,T)$ in the $n$-dimensional hypercube $Q_n$ will satisfy an Erd\"os-S\'os type Conjecture and verify it for some infinite families of trees $T$.
\end{abstract}
\section{Introduction}\label{sec:introduction}
For a forbidden graph $F$ and integer $n$, the \emph{extremal number},
\begin{equation*}
    \ex(n,F):= \max\{e(H) \ : \ H\text{ is }F\text{-free on }n\text{ vertices } \},
\end{equation*}
the largest number of edges in an $F$-free graph, has been studied intensely for the past century. The celebrated Theorem of Erd\"os, Stone, and Simonovits determined the asymptotic behavior of $\ex(n,F)$ in the nonbipartite case.  In this paper we consider the intersection of two active areas of research in extermal numbers: the famed Erd\H{o}s-S\'os Conjecture and rainbow extremal numbers.

\subsection{The Erd\H{o}s-S\'os Conjecture}

When the forbidden graph is a $k$-edge tree $T$, the general Theorem of Erd\"os, Stone, and Simonovits only provides $\ex(n,T)=o(n^2)$. The following Conjecture of Erd\"os and S\'os concerning $\ex(n,T)$ is one of the most well-known in extremal graph theory.

\begin{conj}[Erd\"os-S\'os]\label{conj:ES}
For every $k$-edge connected tree $T$ we have $\ex(n,T)=\floor{\frac{(k-1)}{2}n}$.
\end{conj}

Decades of work have resulted in partial progress on this notorious problem. It is known to be true when $T$ is a path \cite{erdos1959es_path}, caterpillar \cite{perles2017es}, spider of diameter $4$ or less \cite{mclennan2005es,wozniak1996es}, or contains a vertex $v$ adjacent to at least $1/2(k-1)$ leaves of $T$ \cite{sidorenko1989es}. Alternatively, one can ask for conditions beyond the number of edges that enable us to embed $T$. There it is known that if $e(G) > \floor{\frac{(k-1)}{2}n}$ it will contain any $k$-edge tree $T$ if, additionally, $G$ has girth at least $5$ \cite{brandt1996es} or contains no $K_{2,\floor{k/18}}$ \cite{haxell2001extend}. The second result is obtained by Haxell using the powerful extendability method first developed by Friedman and Pippenger in \cite{friedman1987extend}.

\subsection{The Rainbow Extremal Number}

In \cite{keevash2007rainbow}, Keevash, Mubayi, Sudakov, and Verstra\"ete introduced a variant of the extremal number, the \emph{rainbow extremal number}, using edge colorings. Recall that given a graph $G=(V,E)$, a \emph{proper edge coloring} is a map $\phi:E \to \N$ so that $\phi(e_1) \neq \phi(e_2)$ whenever $e_1 \cap e_2 \neq \emptyset$. We say that an edge coloring $\phi$ is \emph{rainbow} on $G$ (or a subgraph $H$) if it is injective to $\N$ (or injective when restricted to $H$). The rainbow extremal number of $F$ as defined in \cite{keevash2007rainbow} is
\begin{equation*}
    \exr(n,F):= \max\{e(H) \ : \ H\text{ on }n\text{ vertices has a proper edge coloring with no rainbow copy of }F \}.
\end{equation*}
In \cite{keevash2007rainbow} the authors showed that $\ex(n,F) \leq \exr(n,F) \leq \ex(n,F)+o(n^2)$, and therefore~$\exr(n,F)$ has been most extensively studied when $F$ is bipartite. For example, in \cite{keevash2007rainbow} the authors show determine the asymptotic $\exr(n,C_6)=\Theta(n^{4/3})$. Johnston, Palmer, and Sarker \cite{johnston2017rainbow} computed $\exr(n,F)$ exactly when $F$ is a forest of stars and give bounds in the case that $F$ is a path (improved subsequently in \cite{ergemlidze2019rainbow} and then again in \cite{johnston2020lower}).

\subsection{A Rainbow Erd\H{o}s-S\'os Conjecture}

Our first objective in this article is to study analogues of Conjecture \ref{conj:ES} for the rainbow Tur\'an number $\exr$. Unfortunately naively replacing $\ex$ with $\ex^*$ in Conjecture \ref{conj:ES} cannot suffice, since previous results have already established that $\exr(n,T)$ can vary for different $k$-edge trees $T$. In particular for the star it is not hard to see that $\exr(n,K_{1,k})=\ex(n,K_{1,k})=\floor{\frac{(k-1)}{2}n}$, since a $k$-edge star is always rainbow in any proper coloring. On the other hand Johnston and Rombach \cite{johnston2020lower} showed that $\exr(n,P_k) \geq \frac{k}{2}n+O(1)$ for the $k$-edge path $P_k$. Therefore no statement as uniform in $T$ as Conjecture \ref{conj:ES} can hold.

Our approach is to change the \emph{host graph}, the latent setting in which $H$ lives, with the aim of restricting the possible behavior of $\exr$ in order to make a rainbow analogue of the Erd\"os-S\'os Conjecture plausible. Formally, given graphs $G$ and $F$, we define the \emph{relative (rainbow) extremal number of $F$ with respect to $G$} as
\begin{align*}
    \ex(G,F)&:= \max\{e(H) \ : \ H\subset G\text{ is }F\text{-free } \}\\
    \exr(G,F)&:= \max\{e(H) \ : \ H\subset G \text{ that has a proper edge coloring with no rainbow copy of }F \}
\end{align*}
and again we see that $\ex(G,F) \leq \exr(G,F) \leq e(G)$, recovering the (rainbow) extremal number when $G=K_n$.
The following quite general question is the motivation for this article.
\begin{ques}\label{ques:ESRhost}
        Given some host graph $G$, does the rainbow extremal number $\exr(G,T)$ depend only on the number of edges in a tree $T$?  Equivalently, which hosts have $\exr(G,T)=\exr(G,K_{1,k})=\ex(G,K_{1,k})$ for all trees $T$ on $k$ edges and all values of $k$?
\end{ques}
\begin{rem}
  It is reasonable to ask this question in the nonrainbow setting as well. Deciding the answer to Question \ref{ques:ESRhost} for $\ex$ and $G=K_n$ is the Erd\"os-S\'os Conjecture.
  The previous work in \cite{brandt1996es} shows that when $G$ has girth at least $5$ it \emph{is} such a host.
  On the other hand it turns out that there is a large class of $G$ for which $\ex(G,K_{1,3})>\ex(G,P_{3})$ and are therefore \emph{not} such hosts. Namely, consider any $G$ which is $K_3$-free and contains the union of two disjoint perfect matchings. Then, $\ex(G,K_{1,3})=|V(G)|$ but taking a subgraph of large minimum degree gives $\ex(G,P_3)\leq |V(G)|-1$. Therefore having witnessed that Conjecture \ref{conj:ES} cannot hold for \emph{all} hosts $G$, it is sensible to continue to focus on the original case $G=K_n$.
\end{rem}
When $G=K_n$ the answer to Question \ref{ques:ESRhost} is ``no'', as discussed above. To find an interesting candidate for a ``yes'' it helps to first understand how to obtain a lower bound on $\exr$ when $G$ contains a disjoint family of perfect matchings. Recall that a perfect matching on $G$ is a collection of edges which meets each vertex exactly once.
    \begin{prop}
        Let $G$ be a graph.
        \begin{enumerate}
            \item We have $\exr(G,K_{1,k})\le \frac{k-1}{2}|V(G)|$ for all $k$.
            \item If $G$ contains $k-1$ disjoint perfect matchings, then $\exr(G,F)\ge  \frac{k-1}{2}|V(G)|$ for every graph $F$ on $k$ edges.
        \end{enumerate}
    \end{prop}
    In light of this Proposition the $n$-dimensional hypercube $Q_n$ is a fairly natural candidate to consider. Formally, $Q_n$ has vertex set $V(Q_n):= \{0,1\}^n$, and an edge between $x=(x_1,\dots,x_n)$ and $y=(y_1,\dots,y_n)$ if $x$ and $y$ differ in exactly one of the $n$ coordinates, i.e. $|\{i \in n |  x_i \neq y_i \} |=1$. We conjecture that a rainbow Erd\H{o}s-S\'os Conjecture holds on $Q_n$.
    \begin{conj}\label{conj:rainbow_cube}
        For each $n \in \N$ we have $\exr(Q_n,T) = \exr(Q_n,K_{1,k})$ for all trees $T$ on $k$ edges (that is, it satisfies Question \ref{ques:ESRhost}).
    \end{conj}
    Recall that $G$ is $1$-factorizable if it can be decomposed into disjoint perfect matchings, and that~$Q_n$ has a natural $1$-factorization using the bit labelings of the edges. A second reason to consider the hypercube $Q_n$ is that the construction in \cite{johnston2020lower} giving a lower bound of $\exr(n,P_k)\geq \frac{k}{2}n+O(1)$ is obtained by adding the `diagonal' edges to $Q_n$, i.e. those of the form $xy$ where $x$ and $y$ differ in all $n$ bits. Therefore, if Conjecture \ref{conj:rainbow_cube} holds, it cannot continue to do so if we add such a diagonal edge. Finally, the relative extremal numbers $\ex(Q_n,F)$ have been intensively studied independently of $\exr$. \cite{chung1992subgraphs,thomason2009bounding,balogh2014cube}.

    Before discussing our partial progress on Conjecture~$\ref{conj:rainbow_cube}$ we will make a second Conjecture analogous to a weakening of the Erd\"os-S\'os Conjecture. Proving that $\ex(n,T)\leq (k-1)n$ (a factor of $2$ from the conjectured truth) is a folklore result, obtained by first deleting vertices until a minimum degree condition is obtained, and then showing that if $T$ has $k$ edges it can be embedded into any graph with minimum degree $k$. Recalling that $\delta(H)$ is the minimum degree of a vertex in $H$, the above motivates defining
\begin{align*}
    \delta(G,F)&:= \max\{\delta(H) \ : \ H\subset G\text{ is }F\text{-free } \}\\
    \deltar(G,F)&:= \max\{\delta(H) \ : \ H\subset G \text{ that has a proper edge coloring with no rainbow copy of }F \}
\end{align*}
using as before the notation $\delta(n,F)$ and $\deltar(n,F)$ when we take $G=K_n$. Figure \ref{fig:ES_table} summarizes the four combinations between uncolored and rainbow embeddings with control on either $e(H)$ or $\delta(H)$. 

\begin{figure}[htb!]
\begin{center}
       \begin{tabular}{c||c|c}

   & Uncolored & Rainbow \\
  \hline
  \hline
  $\delta$ & Folklore & ? \\
  \hline
  $\ex$ & Erd\"os-S\'os conjecture & \cite{keevash2007rainbow,ergemlidze2019rainbow,johnson2017saturated,johnston2017rainbow,johnston2020lower} \\
  
\end{tabular}
\caption{The four problems considered here for a $k$-edge tree $T$}\label{fig:ES_table}
\end{center}
\end{figure}

The natural problem is, in the spirit of Question \ref{ques:ESRhost}, if there are hosts $G$ where $\delta(G,T)$ depends only on the number of edges in $T$. It turns out that $\deltar(n,T)$, like $\exr(n,T)$, is sensitive to the structure of $T$. Andersen conjectured \cite{Andersen1989} that any proper edge-coloring of $K_n$ contains a rainbow path on $n-1$ vertices, after Maamoun and Meyniel \cite{MAAMOUN1984}, disproving a conjecture of Hahn, \cite{Hahn1980}, gave a proper edge-coloring of $K_{2^p}$ with no rainbow Hamiltonian path (for recent progress on Andersen's conjecture see \cite{Alon2017}). In Proposition \ref{prop:delta_n} below we use this construction to show that $\deltar(n,K_{1,k})$ and $\deltar(n,P_k)$ disagree.
We conjecture again that $G=Q_n$ is a good alternative.
\begin{conj}\label{conj:delta_qn}
     $\deltar(Q_n,T) = k-1$ for all trees $T$ on $k$ edges.
\end{conj}

\section{Main Results}\label{sec:results}
   For our results we will refer to $T$ as a forbidden tree on $k$ edges and $k+1$ vertices. Unless we are interested in specific structural information we will enumerate the vertices as $\{x_0,x_1,\dots,x_k\}$, and say that such is a leaf-ordering if for all $1 \leq i \leq k$ the vertex $x_i$ has a unique neighbor preceding it in the ordering; we call this neighbor $x_{i'}$. It is well-known that such an ordering always exists. We will sometimes label the edge $x_{i'}x_i$ as $e_i$. 
   
   Our first main result verifies Conjecture $\ref{conj:rainbow_cube}$ when $T$ is $P_3$, $P_4$, or one of an infinite family of trees with many leaves. We also record a weak upper bound, off by a factor of roughly $4$ from the conjectured truth, which is applicable to any tree.
    \\ \\
    
    \begin{thm}\label{thm:ex_gen_intro}
    We have the following.
    \begin{enumerate}
    \item $\exr(Q_n,P_3)=  2^n$.
    \item $\exr(Q_n,P_4)= \frac{3}{2} 2^n$. \label{item:ex_gen_intro_p5}
    \item $\exr(Q_n,T)= \frac{k-1}{2} 2^n$ for any $k$ edge tree $T$ with a vertex $x_0$ which is adjacent to $\ell> \frac{3}{4}k$ leaves . \label{item:ex_gen_intro_leaves}
    \item $\exr(Q_n,T)<(2k-1)2^n$ for any $k$ edge tree $T$. \label{item:ex_gen_intro_greedy}
    \end{enumerate}
    \end{thm}    
    
Our second main result confirms Conjecture \ref{conj:delta_qn} when $T$ is a long pendant path, or contains many leaves, or is a certain type of spider. We define \emph{spider} as a tree where all but one vertex (the ``root'') has degree at most 2. We call the paths connected to the root vertex the \emph{legs} of the spider. For spiders we will label the root vertex $x_0$ while leg $s$ of length $j$ consists of a path on vertices $\{x_0,x^{(s)}_{1},\dots,x^{(s)}_{j}\}$. Similar to the labeling above, we will use $e_i^{(s)}$ to refer to the edge $x^{(s)}_{i-1}x^{(s)}_{i}$, or $x_0x^{(s)}_{1}$ when $i=1$. A \emph{pendant path} of length $k$ is a tree $T$ that contains an induced path on vertices $\{x_0,x_1,\dots x_k\}$ with $d(x_1)=d(x_2)= ... =d(x_{k-1})=2$ and $d(x_k)=1$. Note that there is no restriction on $d(x_0)$, so this is merely a $k$ edge path attached to a larger tree. We now can state our results. 
\begin{thm}\label{thm:delta_qn}
    We have equality $\deltar(Q_n,T)= k-1$ whenever $T\dots$
    \begin{enumerate}
        \item is a path, or more generally has a pendant path of at least $\frac{3k-1}{4}$ edges, or \label{item:delta_qn_pendant}
        \item is a star, or more generally with at least $\frac{k-1}{2}$ leaves, or \label{item:delta_qn_leaves}
        \item is a spider with legs of even length, or \label{item:delta_qn_espider}
        \item is a spider with legs of length $3$. \label{item:delta_qn_3spider}
    \end{enumerate} 
\end{thm}

\subsection{Paper Organization}
    
The rest of the paper is organized as follows. Since bounds on $\deltar$ imply weak bounds on $\exr$ it is slightly more streamlined to examine $\deltar$ first, and thus in Section \ref{sec:min_degree} we prove Theorem \ref{thm:delta_qn} and other auxillary results where we change the host graph to $K_n$. In Section \ref{sec:ex*_host} we prove a general result for $\exr$ (applying not only to $Q_n$ but many $K_3$-free hosts) which implies Theorem \ref{thm:ex_gen_intro}. 

\section{Rainbow Minimum Degree}\label{sec:min_degree}
In this section we will prove Theorem \ref{thm:delta_qn}, but first we will show, as claimed in Section \ref{sec:introduction}, that $\deltar(n,T)$ depends on the structure of $T$ and not only the number of edges.
\subsection{$K_n$ host}
\begin{prop}\label{prop:delta_n}
    We have that
    \begin{enumerate}
        \item $\deltar(G,T)<2k$ for any $k$ edge tree $T$. \label{item:delta_n_ub}
        \item $\deltar(n,K_{1,k})=k-1$ for the $k$ edge star. \label{item:delta_n_star}
        \item If $k=2^p-1$ for an integer $p\geq 2$ and $ (k+1) | n$, then $\deltar(n,P_{k})\geq k$ for the $k$ edge path. \label{item:delta_n_path}
    \end{enumerate}
\end{prop}
\begin{proof}[Proof of Proposition \ref{prop:delta_n} Part \ref{item:delta_n_ub}]
      Suppose that $G$ is properly colored by $\phi$ and has minimum degree $2k$. Let $\{x_0,\ldots,x_{k}\}$ be a leaf-ordering of the vertices of $T$ and $u_0u_1$ be any edge in $G$.  For $i=2,\dots,k$, inductively assume we have selected distinct vertices $u_0,\ldots,u_{i-1}$ such that for all $1\le j\le i-1$, $G$ contains the edge $u_{j'} u_{j}$ and moreover  $\phi(u_{j_1'}u_{j_1}) \neq \phi(u_{j_2'}u_{j_2})$ whenever $j_1 \neq j_2$.  Choose $u_i$ to be any neighbor of $u_{i'}$ such that
      \begin{enumerate}
          \item $u_i\notin \{u_0,\ldots,u_{i-1}\}$ and
          \item $\phi(u_{i'}u_{i}) \not \in \{\phi(u_0u_1),\dots,\phi(u_{(i-1)'}u_{i-1})\}$
      \end{enumerate}
      Note that at most $(i-1)+(i-1)< 2k$ neighbors of $u_{i'}$ fail one of these conditions, so by the minimum degree of $G$ we can find such a $u_i$. Inductively proceeding through the step $i=k$ gives a rainbow copy of $T$.
\end{proof}

\begin{proof}[Proof of Proposition \ref{prop:delta_n} Part \ref{item:delta_n_star}]
    A $K_{1,k}$ is always rainbow in a proper coloring.
\end{proof}

\begin{proof}[Proof of Proposition \ref{prop:delta_n} Part \ref{item:delta_n_path}]
      While it is well known that even-sized cliques have many $1$-factorizations, identifying one which avoids rainbow paths seems less trivial. The main result of \cite{MAAMOUN1984} is that $K_{2^p}$ can be colored so that it contains no rainbow Hamiltonian path. Therefore when $k=2^p-1$ we have $\deltar(n,P_{k})\geq k$; the result follows by taking unions of disjoint copies of this clique.
\end{proof}

\subsection{$Q_n$ host}

Since the union of $k-1$ matchings (each corresponding to a single coordinate) on $Q_n$ is a $(k-1)$-regular subgraph with a~$(k-1)$-color proper edge coloring, it is immediate that $\deltar(Q_n,T)\geq k-1$ for any $k$ edge tree $T$. Proving equality for various trees $T$ is the focus of this section, proving the various cases of Theorem \ref{thm:delta_qn}.

The basic idea is to pick an ordering of the vertices of $T$ and greedily attempt to embed each vertex one at a time. We need to use our minimum degree condition to avoid
\begin{enumerate}
    \item choosing a vertex already assigned to a vertex of $T$ or
    \item choosing as an edge one whose color is already used on an edge of $T$.
\end{enumerate}
The apparent issue is that minimum degree $k$ is not enough to guarantee both these things, and our approach is to relax the first condition by using the structure of the hypercube. If nearby edges do not share coordinates, the resulting embedding will remain a tree. The following Lemma formalizes this process; we will use it to prove the component parts of Theorem \ref{thm:delta_qn}.
\begin{lem}\label{lem:coordinate_lemma}
    Let $T$ be a $k$ edge tree with $\{x_0,\dots,x_k\}$ a leaf ordering $T$ and edges $e_i=x_{i'}x_i$. Suppose $H_0$ is an auxiliary graph on vertex set $V(H_0)=\{e_1,\dots,e_k\}$ so that for all $j \in [k]$
        \begin{equation*}
            |\{i<j \text{ s.t. } e_ie_j\in E(H_0) \} |\leq k-j\,.
        \end{equation*}
        Then define $H$ as the graph obtained by adding to $H_0$ all edges of the form $e_ie_j$ with $i \neq j$ and $e_i\cap e_j \neq \emptyset$. Suppose furthermore that every path $P$ of length $2\ell$ in $T$, composed of edges $e_{i_1},\dots,e_{i_{2\ell}}$, that the induced subgraph of $H$ on vertices $\{e_{i_1},\dots,e_{i_{2\ell}}\}$ has chromatic number $\chi \geq \ell+1$.

    Then $\deltar(Q_n,T)\leq k-1$.
\end{lem}
\begin{proof}
    Let $G\subset Q_n$ have minimum degree $k$ and proper edge coloring $\phi$.
    First we embed $x_0$ to an arbitrary vertex $v_0$, before iteratively embedding vertex $x_i$ to some vertex $v_i$, for $1 \leq i \leq k$. Let $f_i$ denote the images of $e_i$ under our embedding. To select $v_i$ we choose a neighbor of $v_{i'}$ so that
    \begin{enumerate}
        \item $\phi(v_{i'}v_i) \not \in \{\phi(f_1),\dots,\phi(f_{i-1}) \}$ and
        \item if $e_i$ and $e_j$ are adjacent in $H$,  $c(v_{i'}v_i) \neq c(f_j)$.
    \end{enumerate}
    First we check that selecting such an $v_i$ is possible. We disjointly partition $[i-1] = A \dot{\cup} B$ by setting
    \begin{equation*}
        A:=\{j \in [i-1] \ : \ x_{i'} \in e_j  \}\,,
    \end{equation*}
    the indices whose edges intersect $x_{i'}$, and taking the complement $B = [i-1]\setminus A$. The vertex $v_{i'}$ has degree at least $k$. The first condition eliminates at most $i-1$ neighbors. The second condition eliminates at most $k-i$ edges from $H_0$, and those from $H \setminus H_0$ are in $A$ and already discarded.In total this is
    $(i-1)+(k-i) < k$
    edges and we can always find such a $v_i$.

    It remains to check that when this algorithm concludes selecting the $v_k$ we have a rainbow embedding of $T$. That the embedding is rainbow is evident since no color is used twice. To verify that no cycle has been created, suppose for the sake of contradiction that we obtained a cycle - it must have even length $2\ell$ since $Q_n$ is bipartite. Let the vertices of the cycle be $v_{i_1},\dots,v_{i_{2\ell}}$, with preimages in $T$ $x_{i_1},\dots,x_{i_{2\ell}}$. Whenever $e_ie_j\in E(H)$ our algorithm ensures $c(f_i)\neq c(f_j)$. Therefore any assignment of coordinates to the edges of $f_{i_1},\dots,f_{i_{2\ell}}$ (that is, a proper vertex coloring of this induced subgraph of $H$) requires $\ell+1$ colors. However, a cycle of length $2\ell$ in $Q_n$ must have an even number of edges utilizing each coordinate, and this is a contradiction. 
\end{proof}

While Lemma \ref{lem:coordinate_lemma} may seem overtechnical it will make the proof of Theorem \ref{thm:delta_qn} much smoother, since now we only need to verify the given hypotheses for various classes of trees $T$. As an expository example we first do so in the case that $T$ is the $k$-edge path. Keep the natural ordering $\{x_0,x_1,\dots,x_k\}$, and, for the graph $H_0$, connect edge $e_i$ to the preceding $k-i$ edges in the path (or all, if there are not this many), skipping the immediately preceding edge $e_{i-1}$ which will appear in $H$. Formally, in $H_0$ we draw the edge from $e_i$ to each of $\{e_{j},\dots,e_{i-2}\}$ where $j = \max(2i-k-1,1)$.

Now suppose that $P=\{e_i,\dots,e_{i-1+2\ell}\}$ is a path in $T$ of length $2\ell$. We claim that in $H$ the vertices $\{e_i,\dots,e_{i+\ell}\}$ form an $(\ell+1)$-clique. For $e_r,e_s$ with $i \leq r < s \leq i+\ell$, we have that $e_s$ and $e_r$ are adjacent in $H$ whenever $2s-k-1 \leq r$, which holds since
\begin{equation*}
    2s-k-1 \leq 2i+2\ell-k-1 \leq i+k-k \leq i \leq r
\end{equation*}
so we are done.

The key difference in arguing other trees $T$ will be constructing the ordering and defining $H_0$. We can consider the upper triangular portion of the adjacency matrix of $H$, found in the first panel of Figure \ref{fig:path_figure}. In the second panel we graphically represent the pattern of $1$s in the adjacency matrix using black regions. While not a formal tool, these pictograms will help clarify the more delicate cases remaining in Theorem \ref{thm:delta_qn}. Viewing each column as the edges in $T$ whose coordinates a new edge can avoid, the second panel makes it immediately clear that the first $q=\floor{(k+1)/2}$ edges can be embedded using entirely distinct coordinates, but our strength to do so wanes as we approach the tip of the path. While not a formal tool, these pictograms will help clarify the more delicate cases remaining in Theorem \ref{thm:delta_qn}. It will be important to recall in the other parts of Theorem \ref{thm:delta_qn} that a path in $T$ may not appear consecutively in our edge ordering and this will make casework slightly more complicated. 

\begin{figure}[!tbp]
  \begin{subfigure}[b]{0.4\textwidth}
    \[
A(H_0) = \begin{pmatrix}
 0& 1 & 1 & \cdots & 1 & 1 & 0 & \cdots & 0 & 0 & 0 & 0 \\
  & 0 & 1 & \cdots & 1 & 1 & 0 & \cdots & 0 & 0 & 0 & 0 \\
  &   & 0 & \cdots & 1 & 1 & 1 & \cdots & 0 & 0 & 0 & 0 \\
  &   &   & \ddots & \vdots & \vdots & \vdots & \ddots & \vdots & \vdots & \vdots & \vdots \\
  &   &   &        & 0 & 1 & 1 & \cdots & 0 & 0 & 0 & 0 \\
  &   &   &        &   & 0 & 1 & \cdots & 0 & 0 & 0 & 0 \\
  &   &   &        &   &   & 0 & \cdots & 0 & 0 & 0 & 0 \\
  &   &   &        &   &   &   & \ddots & \vdots & \vdots & \vdots & \vdots \\
  &   &   &        &   &   &   &        & 0 & 1 & 1 & 0 \\
  &   &   &        &   &   &   &        &  & 0 & 1 & 0 \\
  &   &   &        &   &   &   &        &  &  & 0 & 1 \\
  &   &   &        &   &   &   &        &  &  &  & 0 \\
\end{pmatrix}
\]
    \caption{The (upper triangular) adjacency matrix of $H$ for $T=P_{k+1}$}

  \end{subfigure}
  \hfill
  \begin{subfigure}[b]{0.4\textwidth}
    
\begin{tikzpicture}[scale=1.4]
   
    \draw[thick] (0,0) rectangle (4,4);
    
    \filldraw[black] (0,4) -- (4,0) -- (2,4) -- cycle;
\end{tikzpicture}
    \caption{Pictogram representing the adjacency matrix $A(H)$}
  \end{subfigure}
  \caption{}\label{fig:path_figure}
\end{figure}

\begin{proof}[Proof of Theorem \ref{thm:delta_qn} Part \ref{item:delta_qn_pendant}]

    The case that $T$ is a path was handled above so we assume that $T$ is a pendant path with $m\geq \frac{3k-1}{4}$ edges, and list the vertices of the path part as $\{x_0,\dots,x_m\}$. Let $T'$ denote the part of $T$ attached at, say, $x_0$ with $k-m$ edges. Let $\{x_0,\dots,y_{k-m}\}$ be a leaf-ordering of $T'$, and set as above $q=\floor{(k+1)/2}$. For our vertex ordering we take $\{x_0,\dots,x_q,y_1,\dots,y_{k-m},x_{q+1},\dots,x_m\}$. As discussed above this choice of $q$ allows us to make $\{x_0x_1,\dots,x_{q-1}x_q \}$ a clique in $H$. For the edges of $T'$ of the form $\{y_{i'}y_i\}$ we take as edges in $H_0$ the (at most) $i-2$ earlier edges of $T'$ which do not intersect $\{y_{i'}y_i\}$, together with the edge $\{x_0x_1\}$. This required $d_{H_0}(y_{i'}y_i)=i-2+1=i-1$, and this is possible since
    \begin{equation*}
        i-1\leq k-m-1 \leq k/4-3/4 \leq k-\frac{k+1}{2}-(k/4+1/4) \leq k-q-i.
    \end{equation*}
    Finally for every $1\leq i \leq m-q$, we connect in $H_0$ the edge $x_{q+i-1}x_{q+i}$ to all of the edges in the set
    $\{x_{2q+2i-2-2m}x_{2q+2i-1-2m},\dots,x_{q+i-3}x_{q+i-2}\}$. That is, we complete the remainder of the long path exactly as we completed the path above.

    Now let $P$ be a path in $T$ of length $2\ell$. If $P$ is contained within the $T \setminus T'$ then it has $\chi\geq \ell+1$ by identical analysis to the path case. If it is contained within the attached part $T'$ then we see a clique in $H$, so we may assume it intersects both $T'$ as well as $T \setminus T'$. Since both $T'$ and the initial segment of the pendant path form cliques in $H$, we are still immediately done unless $P$ intersects both sides in $\ell$ edges, in which case the $(\ell+1)$st coordinate is guaranteed by the edge $x_0x_1$ which is complete to the edges of $T'$. The adjacency of $H$ is diagrammed in the first panel of Figure $\ref{fig:pendant_leaves}$.
\end{proof}

\begin{proof}[Proof of Theorem \ref{thm:delta_qn} Part \ref{item:delta_qn_leaves}]
     Let $T'$ denote the tree obtained by deleting the leaves of $T$, and suppose $T'$ has $m$ edges remaining. For our ordering, let $\{x_0,x_1,\ldots,x_m\}$ denote an arbitrary ordering of the vertices of $T'$, and $\{x_{m+1},\dots,x_k\}$ be an arbitrary ordering of the leaves of $T$.

     By hypothesis $m \leq k-\frac{k-1}{2}$ and therefore as before we can make the edges of $T'$ a clique in $H$. Denoting the neighbor of $x_{m+i}$ in $T'$ as $x_{(m+i)'}$, then we connect (in $H_0$) edge $x_{(m+i)'}x_{m+i}$ to all edges of the form $x_r x_s$ with $x_rx_{(m+i)'} \in T'$ and $(m+i)'<s \leq m$ and $x_s$ having at least one leaf in $T$. This is permissible under the degree requirement for $H_0$ because each such $x_rx_s$ injects to an as-of-yet unembedded leaf of $T$.

     Now suppose that $P$ is a path of length $2\ell$ in $T$. Since any path can only meet two leaves of $T$, and $T'$ is a clique in $H$, we are immediately done unless $\ell=2$ and we intersect two leaves of $T$. But then the edge in $H_0$ from one leaf edge to the neighbor in $T'$ of the other gives $\chi \geq 3$ as needed. The adjacency of $H_0$ is diagrammed in the second panel of Figure $\ref{fig:pendant_leaves}$.

     \begin{figure}[!tbp]
  \begin{subfigure}[b]{0.4\textwidth}
     
\begin{tikzpicture}[scale=1.4]
   
    \draw[thick] (0,0) rectangle (4,4);

    \filldraw[black] (0,4) -- (2,2) -- (2,4) -- cycle;
    \filldraw[black] (2,3.75) -- (2.5,3.75) -- (2.5,4) -- (2,4) -- cycle;
    \filldraw[black] (2.5,2) -- (2,2) -- (2.5,1.5) -- cycle;
    \filldraw[black] (2.5,1.5) -- (4,0) -- (3.25,1.5) -- cycle;
    \filldraw[black] (2.5,3.25) -- (2.5,2) -- (3.25,2) -- cycle;

    \draw[decorate,decoration={brace,amplitude=10pt,mirror},thick] (2,4.1) -- (0,4.1)
        node[midway,above=10pt] {$q$};
    \draw[decorate,decoration={brace,amplitude=10pt,mirror},thick] (2.5,4.1) -- (2,4.1)
        node[midway,above=10pt] {$T'$};
    \draw[decorate,decoration={brace,amplitude=10pt,mirror},thick] (4,4.1) -- (2.5,4.1)
        node[midway,above=10pt] {$m-q$ left};

    \draw[decorate,decoration={brace,amplitude=10pt,mirror},thick] (0,4.0) -- (0,3.75)
        node[midway,left=10pt] {$x_0x_1$};
    \draw[decorate,decoration={brace,amplitude=10pt,mirror},thick] (0,2) -- (0,1.5)
        node[midway,left=10pt] {$T'$};
\end{tikzpicture}
    \caption{Pictogram for $T$ a pendant path.}
  \end{subfigure}
  \hfill
   \begin{subfigure}[b]{0.4\textwidth}
     
\begin{tikzpicture}[scale=1.4]
    
    \draw[thick] (0,0) rectangle (4,4);

    \filldraw[black] (0,4) -- (1.5,2.5) -- (1.5,4) -- cycle;
    \filldraw[black] (1.5,3.5) -- (2,3.5) -- (2,3.75) -- (1.5,3.75) -- cycle;
    \filldraw[black] (2,3.25) -- (2.5,3.25) -- (2.5,3.5) -- (2,3.5) -- cycle;
    \filldraw[black] (3.25,2.5) -- (3.75,2.5) -- (3.75,2.75) -- (3.25,2.75) -- cycle;
    \filldraw[black] (3.25,3) -- (3.75,3) -- (3.75,3.25) -- (3.25,3.25) -- cycle;

      \draw[decorate,decoration={brace,amplitude=10pt,mirror},thick] (1.5,4.1) -- (0,4.1)
        node[midway,above=10pt] {$m$};
     \draw[decorate,decoration={brace,amplitude=10pt,mirror},thick] (4,4.1) -- (1.5,4.1)
        node[midway,above=10pt] {leaves};
        
\end{tikzpicture}
    \caption{Pictogram for $T$ with many leaves.}
  \end{subfigure}
  \caption{}\label{fig:pendant_leaves}
\end{figure}

\end{proof}

\begin{proof}[Proof of Theorem \ref{thm:delta_qn} Part \ref{item:delta_qn_espider}]
     Suppose the legs of the given spider have length $2k_1,\ldots,2k_t$. We will label the vertices of the spider as $x_0$ the root together with $x_i^{(s)}$, where $1 \leq s \leq t$ and $1 \leq i \leq 2k_s$. For our ordering we take $\{x_0,x_1^{(1)},\dots,x_{k_1}^{(1)},x_{1}^{(2)},\dots,x_1^{(s)},\dots,x_{k_s}^{(s)},x_{k_1+1}^{(1)},\dots,x_{2k_1}^{(1)},\dots,x_{2k_1}^{(s)},\dots,x_{2k_s}^{(s)}\}$. Therefore the first $k_1+\dots+k_s$ edges, the first half of each leg closest to the root, may be taken to be a clique in $H$. For each $s=1,\dots, t$, and each $i=k_1+1,\dots,2k_s$, in $H_0$ we connect the edge $ x_{i-1}^{(s)}x_i^{(s)}$ to the edges $\{x^{(s)}_{2i-2k_s-2}x^{(s)}_{2i-2k_s},\dots,x^{(s)}_{i-3}x^{(s)}_{i-2} \}$ on the same arm (analogous to the path case), as well as the edges $\{x_0x^{(r)}_1,\dots,x^{(r)}_{k_s-1}x^{(r)}_{k_s}\}$ for $r>s$. This is a total of
     \begin{equation*}
         (i-2i+2k_s)+k_{s+1}+\dots+k_t = 2k_s-i+\sum_{r=s+1}^{t}k_r
     \end{equation*}
    edges, satisfying the degree condition of Lemma \ref{lem:coordinate_lemma}.

    Now let $P$ be a path of length $2\ell$. If $P$ intersects only one arm of the spider we are done by identical analysis to the path case. Suppose then that $P$ intersects two arms, say arms $r$ and $s$ with $a$ and $b$ edges on arm $r$ and arm $s$ respectively. Since the first $k_r$ and $k_s$ edges of these two arms have totally distinct coordinates, and $a+b = \ell$, we are clearly done unless $a>k_r$ or $b>k_s$. On the other hand since $a\leq 2k_r$ and $b \leq 2k_s$ we are also done by the clique of size $k_r+k_s$ unless we are in the extreme case $a=2k_r$ and $b=2k_s$. In this case $P$ sees two entire arms and the fact that the second $k_s$ edges of arm $s$ are complete to the $k_r$ clique from the first arm suffices to finish. The adjacency of $H$ is diagrammed in the first panel of Figure $\ref{fig:spiders}$.
\end{proof}
\begin{proof}[Proof of Theorem \ref{thm:delta_qn} Part \ref{item:delta_qn_3spider}]

    We use very similar notation to the previous problem, with vertices $x_0$ and $x^{(s)}_1,x^{(s)}_2,x^{(s)}_3$ for $1 \leq s \leq t$. For our ordering we take $\{x_0,x^{(1)}_1,x^{(2)}_1,\dots,x^{(t)}_1,x^{(1)}_2,\dots,x^{(t)}_2,x^{(t)}_3,\dots,x^{(1)}_3\}$. Note that the outermost $t$ edges will be embedded in reverse.

    As we have seen before, for $m=\floor{\frac{3t+1}{2}}-t$ we can make the first $t+m$ edges a clique in $H$ (the $t$ shift will be convenient for indexing). For the remaining edges in the second row, $\{x^{(m+1)}_1x^{(m+1)}_2,\dots,x^{(t)}_1x^{(t)}_2 \}$, we connect in $H_0$ edges $ x^{(i)}_1x^{(i)}_2$ to all $t$ edges in the bottom row (with the exception of $x_0x^{(i)}_1$ which will appear in $H$) as well as the first $t-i$ edges of the second row, namely $\{x^{(1)}_1x^{(1)}_2,\dots,x^{(t-i)}_1x^{(t-i)}_2 \}$, for $1 \leq i \leq t$. For an edge in the third layer $x^{(i)}_2x^{(i)}_3$ we connect it in $H_0$ to $m+t-i$ edges in the middle row, namely $\{x^{(i-t+1)}_1x^{(i-t+1)}_2,\dots,x^{(m)}_1x^{(m)}_2\}$ as well as the $t-m$ edges from the first row, $\{x_0x^{(m+1)}_1,\dots,\dots,x_0x^{(t)}_1\}$. The remaining $m$ edges $x^{(i)}_2x^{(i)}_3$ for $1 \leq i \leq m$ we do not connect to anything. These choices are made so that for an arm $s$ with $m <s$ has, for each arm $r\leq m$, either the second or third edge of arm $s$ using different coordinates than the second edge in arm $r$.

    Now let $P$ be a path of length $2\ell$. Since $P$ passes through $x_0$ at most once we may assume that it is contained in two arms, $\{x_0,x^{(r)}_1,x^{(r)}_2,x^{(r)}_3 \}$ and $\{x_0,x^{(s)}_1,x^{(s)}_2,x^{(s)}_3 \}$ with $1 \leq r < s \leq t$. We consider cases based on the values of $r$ and $s$. In each case we must verify that every induced subgraph on $4$ consecutive vertices has $\chi\geq 3$ and the subgraph induced on all $6$ vertices has $\chi \geq 4$. Proceeding from top to bottom in Figure \ref{fig:spider_arms} we see these $6$-vertex graphs for the four cases,
    \begin{enumerate}
        \item $r,s\leq m$,
        \item $m <r,s$,
        \item $r\leq m < s$ with $r \leq t-s $,
        \item and finally $r\leq m < s$ with $r>t-s$.
    \end{enumerate}
    In each case verifying the required property for $\chi$ can be done by inspection.

    The adjacency of $H_0$ is diagrammed in the second panel of Figure $\ref{fig:spiders}$. 
    
    \begin{figure}[!tbp]
  \begin{subfigure}[b]{0.4\textwidth}
     
\begin{tikzpicture}[scale=1.4]
    
    \draw[thick] (0,0) rectangle (4,4);

    \filldraw[black] (0,4) -- (2,2) -- (2,4) -- cycle;
    \filldraw[black] (2,4) -- (2.5,3) -- (3,3) -- (3,2) -- (2.5,2) -- (3,1) -- (2,2) -- cycle;
    \filldraw[black] (3,3) -- (3.25,2.5) -- (3.5,2.5) -- (3.5,2) -- (3,2) -- cycle;
    \filldraw[black] (3,1) -- (3.25,1) -- (3.5,0.5) -- cycle;
    \filldraw[black] (3.5,0.5) -- (3.75,0.5) -- (4,0) -- cycle;
    \filldraw[black] (3.5,2.5) -- (3.75,2) -- (3.5,2) -- cycle;

    \draw[decorate,decoration={brace,amplitude=10pt,mirror},thick] (2,4.1) -- (0,4.1)
        node[midway,above=10pt] {$m$};
    \draw[decorate,decoration={brace,amplitude=10pt,mirror},thick] (3,4.1) -- (2,4.1)
        node[midway,above=10pt] {$k_1$};
    \draw[decorate,decoration={brace,amplitude=10pt,mirror},thick] (3.5,4.1) -- (3,4.1)
        node[midway,above=10pt] {$k_2$};
    \draw[decorate,decoration={brace,amplitude=10pt,mirror},thick] (4,4.1) -- (3.5,4.1)
        node[midway,above=10pt] {$k_3$};
\end{tikzpicture}
    \caption{Pictogram for $T$ an even spider.}
  \end{subfigure}
  \hfill
   \begin{subfigure}[b]{0.4\textwidth}
     
\begin{tikzpicture}[scale=1.4]
    
    \draw[thick] (0,0) rectangle (4,4);

    \filldraw[black] (0,4) -- (2,2) -- (2,4) -- cycle;
    \filldraw[black] (2,4) -- (2.66,4) -- (2.66,2.66) -- (2,2) -- cycle;
    \filldraw[black] (2.66,4) -- (2.66,3.33) -- (3.333,3.33) -- cycle;
    \filldraw[black] (2.66,3.333) -- (2.66,2.666) -- (3.333,2.666) -- (3.333,3.333) -- cycle;

   \draw[decorate,decoration={brace,amplitude=10pt,mirror},thick] (2,4.1) -- (0,4.1)
        node[midway,above=10pt] {$t+m$};
    \draw[decorate,decoration={brace,amplitude=10pt,mirror},thick] (2.666,4.1) -- (2,4.1)
        node[midway,above=10pt] {$t-m$};
    \draw[decorate,decoration={brace,amplitude=10pt,mirror},thick] (4,4.1) -- (2.666,4.1)
        node[midway,above=10pt] {$t$};

    \draw[decorate,decoration={brace,amplitude=10pt,mirror},thick] (0,4.0) -- (0,2.666)
        node[midway,left=10pt] {$t$};
    \draw[decorate,decoration={brace,amplitude=10pt,mirror},thick] (0,2.666) -- (0,2)
        node[midway,left=10pt] {$m$};
\end{tikzpicture}
    \caption{Pictogram for $T$ a $3$-spider.}
  \end{subfigure}
  \caption{}\label{fig:spiders}
\end{figure}

\begin{figure}
\begin{center}
    
\begin{tikzpicture}

\coordinate (v1) at (-6,1.5);  
\coordinate (v2) at (-4,0.5);  
\coordinate (v3) at (-1.5,0);  
\coordinate (v5) at (1.5,0);  
\coordinate (v6) at (4,0.5);   
\coordinate (v7) at (6,1.5);

\filldraw (v1) circle (2pt) node[below] {$x^{(r)}_3x^{(r)}_2$};
\filldraw (v2) circle (2pt) node[below] {$x^{(r)}_2x^{(r)}_1$};
\filldraw (v3) circle (2pt) node[below] {$x^{(r)}_1x_0$};
\filldraw (v5) circle (2pt) node[below] {$x_0x^{(s)}_1$};
\filldraw (v6) circle (2pt) node[below] {$x^{(s)}_1x^{(s)}_2$};
\filldraw (v7) circle (2pt) node[below] {$x^{(s)}_2x^{(s)}_3$};

\draw (v1) -- (v2);
\draw (v2) -- (v3);
\draw (v3) -- (v5);
\draw (v5) -- (v6);
\draw (v6) -- (v7);

\draw(v2) -- (v5);
\draw(v2) -- (v6);
\draw(v3) -- (v6);
\draw(v2) -- (v6);

\end{tikzpicture}

\begin{tikzpicture}

\coordinate (v1) at (-6,1.5);  
\coordinate (v2) at (-4,0.5);  
\coordinate (v3) at (-1.5,0);  
\coordinate (v5) at (1.5,0);  
\coordinate (v6) at (4,0.5);   
\coordinate (v7) at (6,1.5);

\filldraw (v1) circle (2pt) node[below] {$x^{(r)}_3x^{(r)}_2$};
\filldraw (v2) circle (2pt) node[below] {$x^{(r)}_2x^{(r)}_1$};
\filldraw (v3) circle (2pt) node[below] {$x^{(r)}_1x_0$};
\filldraw (v5) circle (2pt) node[below] {$x_0x^{(s)}_1$};
\filldraw (v6) circle (2pt) node[below] {$x^{(s)}_1x^{(s)}_2$};
\filldraw (v7) circle (2pt) node[below] {$x^{(s)}_2x^{(s)}_3$};

\draw (v1) -- (v2);
\draw (v2) -- (v3);
\draw (v3) -- (v5);
\draw (v5) -- (v6);
\draw (v6) -- (v7);

\draw(v1) -- (v3);
\draw(v1) -- (v5);
\draw(v2) -- (v5);

\draw(v3) -- (v6);
\draw(v7) -- (v3);
\draw(v7) -- (v5);

\end{tikzpicture}

\begin{tikzpicture}

\coordinate (v1) at (-6,1.5);  
\coordinate (v2) at (-4,0.5);  
\coordinate (v3) at (-1.5,0);  
\coordinate (v5) at (1.5,0);   
\coordinate (v6) at (4,0.5);   
\coordinate (v7) at (6,1.5);

\filldraw (v1) circle (2pt) node[below] {$x^{(r)}_3x^{(r)}_2$};
\filldraw (v2) circle (2pt) node[below] {$x^{(r)}_2x^{(r)}_1$};
\filldraw (v3) circle (2pt) node[below] {$x^{(r)}_1x_0$};
\filldraw (v5) circle (2pt) node[below] {$x_0x^{(s)}_1$};
\filldraw (v6) circle (2pt) node[below] {$x^{(s)}_1x^{(s)}_2$};
\filldraw (v7) circle (2pt) node[below] {$x^{(s)}_2x^{(s)}_3$};

\draw (v1) -- (v2);
\draw (v2) -- (v3);
\draw (v3) -- (v5);
\draw (v5) -- (v6);
\draw (v6) -- (v7);

\draw(v2) -- (v5);
\draw(v2) -- (v6);
\draw(v3) -- (v6);
\draw(v5) -- (v7);

\end{tikzpicture}

\begin{tikzpicture}

\coordinate (v1) at (-6,1.5);  
\coordinate (v2) at (-4,0.5);  
\coordinate (v3) at (-1.5,0);  
\coordinate (v5) at (1.5,0);   
\coordinate (v6) at (4,0.5);   
\coordinate (v7) at (6,1.5);

\filldraw (v1) circle (2pt) node[below] {$x^{(r)}_3x^{(r)}_2$};
\filldraw (v2) circle (2pt) node[below] {$x^{(r)}_2x^{(r)}_1$};
\filldraw (v3) circle (2pt) node[below] {$x^{(r)}_1x_0$};
\filldraw (v5) circle (2pt) node[below] {$x_0x^{(s)}_1$};
\filldraw (v6) circle (2pt) node[below] {$x^{(s)}_1x^{(s)}_2$};
\filldraw (v7) circle (2pt) node[below] {$x^{(s)}_2x^{(s)}_3$};

\draw (v1) -- (v2);
\draw (v2) -- (v3);
\draw (v3) -- (v5);
\draw (v5) -- (v6);
\draw (v6) -- (v7);

\draw(v2) -- (v5);
\draw(v3) -- (v6);
\draw(v5) -- (v7);
\draw(v3) -- (v7);

\end{tikzpicture}

\caption{From top to bottom: 1) $r,s\leq m$ 2) $m <r,s$ 3) $r\leq m < s$ with $r \leq t-s $ 4) $r\leq m < s$ with $r>t-s$}\label{fig:spider_arms}
\end{center}
\end{figure}
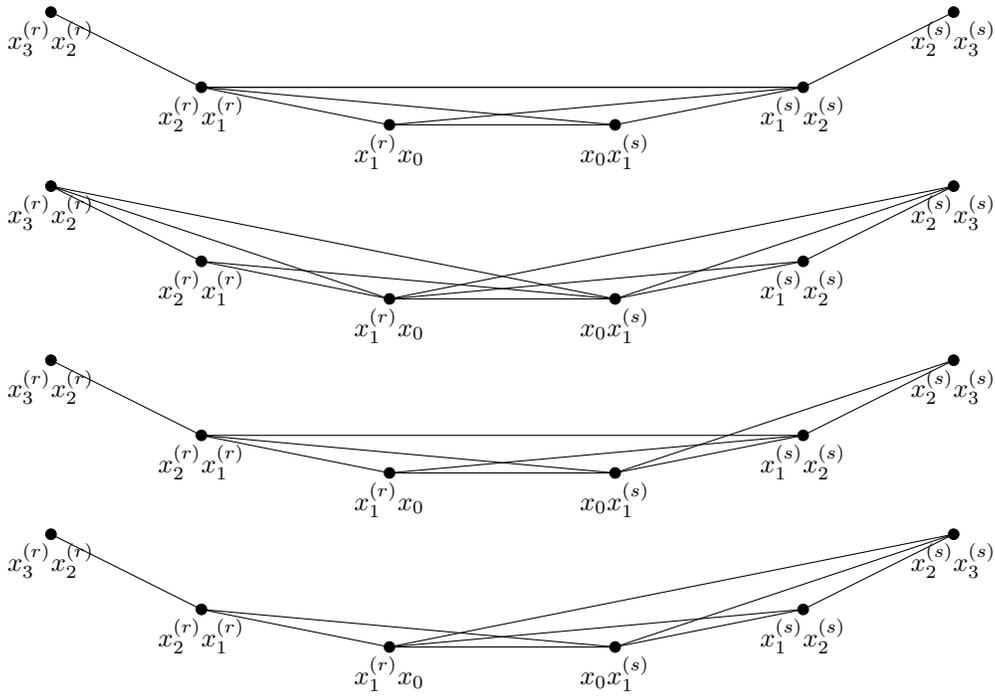

\end{proof}

\newpage

\section{Rainbow Extremal Number}\label{sec:ex*_host}
 Our results on $\exr$ will apply to $Q_n$ but also extend to a broad class of graphs and generally only depend on certain forbidden subgraphs. Namely we will prove the following, which gives Theorem \ref{thm:ex_gen_intro} since $Q_n$ satisfies the hypotheses of each result.

 \begin{thm}\label{thm:ex_gen}
    We have that
    \begin{enumerate}
    \item If $G$ is any graph and $T$ is a $k$-edge tree then $\exr(G,T)<(2k-1)|V(G)|$. \label{item:ex_gen_greedy}
    \item If $G$ is $K_3$-free then $\exr(G,P_3)\le  |V(G)|$.\label{item:ex_gen_p4}
    \item If $G$ is $\{K_3,K_{2,3}\}$-free then $\exr(G,P_4)\le \frac{3}{2} |V(G)|$. \label{item:ex_gen_p5}
    \item If $r \geq 3$ and $T$ is a $k$-edge tree with a vertex $x_0\in V(T)$ that is adjacent to $\ell> \frac{3}{4}k+\frac{r-3}{2}$ leaves, then any $K_{2,r}$-free graph $G$ has $\exr(G,T)\le \frac{k-1}{2}|V(G)|$. \label{item:ex_gen_leaves}
    \end{enumerate}
    \end{thm}
    
    It will prove useful to enforce a minimum degree condition. For a subset of vertices $S \subset V$ let $e_S:=e(S,S)+e(S,V\setminus S)$ denote the number of edges which intersect $S$ in either $1$ or $2$ vertices. The following Lemma is folklore but we include a proof for completeness.

 \begin{lem}\label{lem:min_degree_subgraph}
    Every nonempty graph $H$ with average degree $d=2e(H)/v(H)$ contains a subgraph $H'$ so that $H'$ has average degree at least $d$ and every $S \subsetneq V(H)$ has $e_S >d/2|S|$. In particular by letting $S$ be a single vertex we see that $\delta(H') >d/2$.
\end{lem}
\begin{proof}
     We iteratively delete sets of verties $S$ with $e_S \leq d/2|S|$. Each time we do so vertex we see a new average degree $\frac{2e(G)-e_s}{v(G)-|S|}\geq \frac{dv(G)-d|S|}{v(G)-|S|}=d$. Therefore the average degree can only increase until this process halts, when we are left with a nonempty graph as required.
\end{proof}
 
\begin{proof}[Proof of Theorem \ref{thm:ex_gen} Part \ref{item:ex_gen_greedy}]
    If $G' \subseteq G$ has $(2k-1)|V(G)|$ edges then Lemma \ref{lem:min_degree_subgraph} gives a subgraph $G'' \subset G'$ with $\delta(G'') > 2k-1$. Then Proposition \ref{prop:delta_n} Part \ref{item:delta_n_ub} implies that we can find a rainbow copy of $T$ in $G''$. 
\end{proof}

\begin{proof}[Proof of Theorem \ref{thm:ex_gen} Part \ref{item:ex_gen_p4}]
We prove that If $G$ is $K_3$-free then $\exr(G,P_3)\leq|V(G)|$.

Any proper edge coloring of the fork graph (Figure~\ref{fig:fork graph}) contains a rainbow-$P_3$. Let $G'\subseteq G$ be properly edge colored with no rainbow $P_3$ - we will bound $e(G')$. Suppose some $x\in V(G')$ has degree at least $3$. Then since $G$ contains no $K_3$ and $G'$ contains no fork, each of the neighbors of $x$ must have degree $1$. It is always possible to decompose $G'=C_1\dot{\cup} \dots \dot{\cup} C_t$ as the disjoint union of $t$ connected components for $t \geq 1$. By the analysis above, if any $C_i$ contains a vertex of degree $3$ or more then $e(C_i) \leq v(C_i)-1$; furthermore  each $C_i$ contains at most one vertex of degree $3$ or more. If $a$ is the number of vertices of degree $3$ or more in $G'$ then
\begin{equation*}
    e(G')=\sum_{s=1}^{t}e(C_i)\leq |V(G)|-a \leq |V(G)|
\end{equation*}
and therefore $ \exr(G,P_3)\leq |V(G)|$.
\end{proof}

\begin{figure}[h!]
\centering
    \begin{tikzpicture}
    
    \node[fill, circle, inner sep=1.5pt, label=above:] (u) at (0,0) {};
    \node[fill, circle, inner sep=1.5pt, label=left:] (v1) at (-2,0) {};
    \node[fill, circle, inner sep=1.5pt, label=above:] (v2) at (2,1) {};
    \node[fill, circle, inner sep=1.5pt, label=below:] (v4) at (2,-1) {};

    \draw (u) -- (v1) node[midway, above] {};
    \draw (u) -- (v2) node[midway, above right] {};
    \draw (u) -- (v4) node[midway, below right] {};

    \node[fill, circle, inner sep=1.5pt, label=above:] (w11) at (-4, 0) {};
    \draw (v1) -- (w11);

\end{tikzpicture}
\caption{The fork graph}\label{fig:fork graph}
\end{figure}
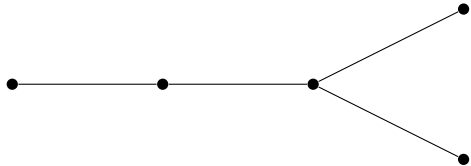

\subsection{$P_4$}

\begin{proof}[Proof of Theorem \ref{thm:ex_gen} Part \ref{item:ex_gen_p5}]

The general shape of the argument is similar to the one above for $P_3$. We begin with a lemma showing that certain arrangements of large-degree vertices contain a rainbow $P_4$.
\begin{lem} \label{lem:4_3path}
    If $G$ is $K_3$-free and contains a path $uvw$ with $\deg(u)\ge 4$ and $\deg(w)\ge 3$, then any proper coloring of $G$ contains a rainbow $P_4$.
\end{lem}

\begin{proof}
    By assumption $u$ has three neighbors $x_1,x_2,x_3 \ne v$ and $w$ has two neighbors $y_1,y_2 \ne v$. Furthermore, as $G$ is triangle-free, $x_1,x_2,x_3 \ne w$ and $y_1,y_2 \ne u$. Suppose that $uv$ is colored $A$ and $vw$ is colored $B$.  As $G$ is properly colored, at least one of $\{wy_1,wy_2\}$ has a color $C \notin \{A,B\}$, and at least one edge in $\{ux_1,ux_2,ux_3\}$ has some color $D \notin \{A,B,C\}$. Without loss of generality, suppose that $wy_1$ is colored $C$. If $x_iu$ has color $D \notin \{A,B,C\}$ and $x_i \ne y_1$ for some $i$ then we are done. Therefore we are done unless no such $x_i$ exists. Note that as $G$ is properly edge-colored and $uv$ is colored $A$, no $x_iu$ is colored $A$. Without loss of generality, we may assume that $x_1=y_1$, $x_2u$ is colored $B$ and $x_3u$ is colored $C$.  Then, $x_1u$ is colored some $D \notin \{A,B,C\}$. Note that $wy_2$ must be colored $E \in \{B,D\}$, as otherwise the path $x_2ux_1wy_2$ will be rainbow with colors $B,D,C,E$.  We must in fact have $E=D$ since $wy_2$ is incident to $wv$.  But in this case the path $y_2wvux_3$ is rainbow with colors $D,B,A,C$.
\end{proof}

    Next we proceed to the proof proper of Theorem \ref{thm:ex_gen} Part \ref{item:ex_gen_p5}. Suppose we have a properly colored $H\subseteq G$ with strictly more than $\frac{3}{2}|V(G)|$ edges; we will show it contains a rainbow $P_4$. First we pass to a subgraph $G'\subseteq H$ of average degree greater than 3 such that every neighbor $v$ of a vertex $u$ of degree at least 4 has degree at least 3.
    
        This is possible by taking $G'$ from Lemma \ref{lem:min_degree_subgraph}.  Note that $G'$ has no vertices of degree less than 2, so the claim can only fail if a vertex $u$ of degree at least $4$ is adjacent to a vertex $v$ of degree 2.  By the furthermore part of Lemma \ref{lem:min_degree_subgraph}, $G'$ contains no two adjacent vertices of degree 2 (since they are incident only to 3 edges), so $v$'s other neighbor must have degree at least 3, in which case Lemma \ref{lem:4_3path} furnishes a rainbow $P_4$. Throughout, when we write $v_1W$ we mean the set of edges between vertex $v_1$ and the set $W$.
    
    Since $G'$ has average degree larger than 3, there exists at least one such $u$ with at least 4 neighbors $v_1,\ldots,v_4$ where each $v_i$ is adjacent to at least two vertices $W_i=\{w_{i,1},w_{i,2}\}$ (note that the sets $W_i$ could at this point intersect each other, but they cannot contain $u$ or the other $v_i$ since $G$ is $K_3$-free). We will show that any such local structure contains a rainbow $P_4$.

    Let's first suppose that one such $W$, say $W_1$, does not intersect any of the others. This is diagrammed in Figure \ref{fig:tree_diagram}. Then if $uv_1$ has color $A$, and $uv_2,uv_3,uv_4$ have colors $B,C,D$, then at least one of $\{B,C,D\}$, say $B$, does not appear on the edges $v_1W$. It is possible that one of the edges $v_2W_2$ has color $A$, but at least one of them avoids both $A$ and $B$. Finally one of the two edges $v_1W$ (which does not use $A$ or $B$) gives the fourth color, distinct from $C$ obtained from $v_2W_2$.

    Then we can suppose that all $W_i$ intersect at least one other, otherwise we may apply the above argument. Furthermore, if we find three $W_i$ meeting on a common vertex, then we would find a vertex of degree $3$ at distance $2$ from one with degree $4$ and we would apply Lemma \ref{lem:4_3path} again. Therefore we may assume that each $W_i$ can intersect at most two other $W_j$. The casework proceeding is easiest understood through the auxiliary graph on vertex set $\{1,2,3,4\}$, where an edge $ij$ appears if $W_i \cap W_j \neq \emptyset$. What we have established so far is that, on this graph, no vertex is isolated and no vertex has degree larger than $2$. Since there are no isolated vertices there are at least $2$ edges, and the degree hypothesis implies that there are at most $4$ edges. The unique graphs satisfying these requirements are the disjoint matching and the $C_4$. Since finding a rainbow path is strictly more difficult after identifying vertices, we are left to check the two maximal graphs with the largest number of gluings; one where $|W_1\cap W_2|=1,|W_3 \cap W_4| = 1$ (the matching) and one where $|W_1 \cap W_2| =1,|W_2 \cap W_3| =1,|W_3 \cap W_4| =1,|W_4\cap W_1| =1$ (the $C_4$). We use here $i \neq j \implies W_i \neq W_j$ since $G$ is $K_{2,3}$-free.

    Let us start with the $C_4$ case, seen in Figure \ref{fig:cycle_diagram}. Suppose first that the color $C$ does not appear on $v_1W_1$; then we find a rainbow $P_4$ by taking extending the path $v_1uv_3$ from $v_3$ by whichever edge does not have color $A$, and then from $v_1$ by taking whichever edge avoids this new color (we avoided $C$ by hypothesis). Then suppose, say, that $v_1w_{1,2}$ is colored $C$. Then consider the path $w_{1,2}v_1uv_4$ which has colors $C,A,D$. If either edge of $v_4W_4$ avoids colors $A,C$ we find a rainbow $P_4$; otherwise both appear and we find a rainbow $P_4$ by applying the argument just given, but now to the antipodal pair $B,D$ since no edge at $D$ has color $B$.
\begin{figure}
\centering
    \begin{tikzpicture}
   
    \node[fill, circle, inner sep=1.5pt, label=above:$u$] (u) at (0,0) {};
    \node[fill, circle, inner sep=1.5pt, label=left:$v_1$] (v1) at (-2,0) {};
    \node[fill, circle, inner sep=1.5pt, label=above:$v_2$] (v2) at (2,1) {};
    \node[fill, circle, inner sep=1.5pt, label=right:$v_3$] (v3) at (2,0) {};
    \node[fill, circle, inner sep=1.5pt, label=below:$v_4$] (v4) at (2,-1) {};

    \draw (u) -- (v1) node[midway, above] {$A$};
    \draw (u) -- (v2) node[midway, above right] {$B$};
    \draw (u) -- (v3) node[midway, right] {$C$};
    \draw (u) -- (v4) node[midway, below right] {$D$};

    \node[fill, circle, inner sep=1.5pt, label=above:$w_{1,1}$] (w11) at (-3, 1) {};
    \node[fill, circle, inner sep=1.5pt, label=below:$w_{1,2}$] (w12) at (-3, -1) {};
    \draw (v1) -- (w11);
    \draw (v1) -- (w12);

    \node[fill, circle, inner sep=1.5pt, label=above:$w_{2,1}$] (w21) at (3, 2) {};
    \node[fill, circle, inner sep=1.5pt, label=above:$w_{2,2}$] (w22) at (4, 1) {};
    \draw (v2) -- (w21);
    \draw (v2) -- (w22);

    \node[fill, circle, inner sep=1.5pt, label=below:$w_{3,1}$] (w31) at (3, -0.5) {};
    \node[fill, circle, inner sep=1.5pt, label=above:$w_{3,2}$] (w32) at (3, 0.5) {};
    \draw (v3) -- (w31);
    \draw (v3) -- (w32);

    \node[fill, circle, inner sep=1.5pt, label=below:$w_{4,1}$] (w41) at (3, -2) {};
    \node[fill, circle, inner sep=1.5pt, label=below:$w_{4,2}$] (w42) at (4, -1) {};
    \draw (v4) -- (w41);
    \draw (v4) -- (w42);
    
\end{tikzpicture}
\caption{The case when $W_1$ does not meet any other $W_i$. Note that the $w$ vertices on the right hand side may be identified if they are attached to distinct $v_i$.}\label{fig:tree_diagram}
\end{figure}
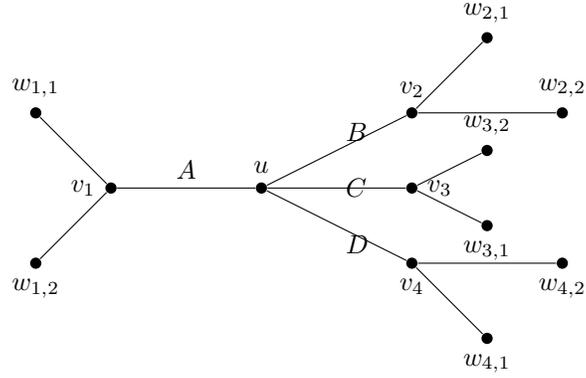

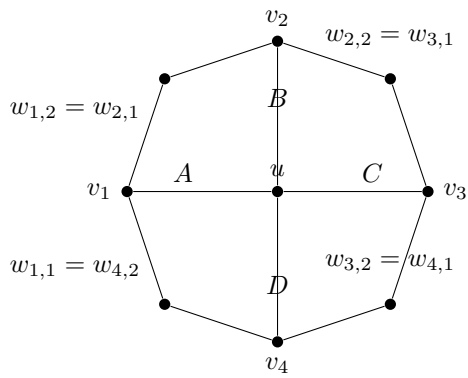
\begin{figure}
\centering
    \begin{tikzpicture}

    \node[fill, circle, inner sep=1.5pt, label=above:$u$] (u) at (0,0) {};
    \node[fill, circle, inner sep=1.5pt, label=left:$v_1$] (v1) at (-2,0) {};
    \node[fill, circle, inner sep=1.5pt, label=above:$v_2$] (v2) at (0,2) {};
    \node[fill, circle, inner sep=1.5pt, label=right:$v_3$] (v3) at (2,0) {};
    \node[fill, circle, inner sep=1.5pt, label=below:$v_4$] (v4) at (0,-2) {};

    \draw (u) -- (v1) node[midway, above left] {$A$};
    \draw (u) -- (v2) node[midway, above] {$B$};
    \draw (u) -- (v3) node[midway, above right] {$C$};
    \draw (u) -- (v4) node[midway, below] {$D$};

    \node[fill, circle, inner sep=1.5pt, label={[label distance=2mm]above left:{$w_{1,1} = w_{4,2}$}}] (w11) at (-1.5,-1.5) {};
    \node[fill, circle, inner sep=1.5pt, label={[label distance=2mm]below left:{$w_{1,2} = w_{2,1}$}}] (w12) at (-1.5,1.5) {};
    
    \node[fill, circle, inner sep=1.5pt, label={[label distance=2mm]above:{$w_{2,2} = w_{3,1}$}}] (w21) at (1.5,1.5) {};
    \node[fill, circle, inner sep=1.5pt, label={[label distance=2mm]above:{$w_{3,2} = w_{4,1}$}}] (w31) at (1.5,-1.5) {};

    \draw (v1) -- (w11);
    \draw (v1) -- (w12);
    
    \draw (v2) -- (w12);
    \draw (v2) -- (w21);
    
    \draw (v3) -- (w21);
    \draw (v3) -- (w31);
    
    \draw (v4) -- (w31);
    \draw (v4) -- (w11);
    
\end{tikzpicture}

\caption{The case when the $W_i$ intersect cyclically.}\label{fig:cycle_diagram}
\end{figure}

\end{proof}

\subsection{Many Leaves}
Suppose that $T$ is a tree with a vertex $x_0$ which is adjacent to many leaves. To find a rainbow embedding of $T$ it will help us to first embed the non-leaves so that they are not adjacent in $G$ to the image of $x_0$ - then we will be generally safe to greedily embed the leaves of $x_0$ without worrying that we have already used these vertices. First we prove a variant of the standard tree embedding lemma, which assumes that the host is $K_{2,r}$-free to ensure some separation between vertices of the embedded tree.
\begin{lem}\label{lem:k2rfree_embedding}
    Let $r \geq 3$ and $T$ be a tree with $\{x_0,\dots,x_k\}$ a leaf ordering. If $G$ is $K_{2,r}$-free with minimum degree $\delta(G)>2k+r-3$ and properly colored, then for all $v_0\in V(G)$ there exist distinct vertices $v_1,v_2,\ldots,v_k$ such that
    \begin{enumerate}
        \item $x_ix_j \in E(T)$ implies $v_iv_j \in E(G)$ and
        \item each of these $k$ edges receives a distinct color and
        \item $v_i$ is adjacent to $v_0$ if and only if $x_i$ is adjacent to $x_0$.
    \end{enumerate}
\end{lem}

\begin{proof}
   Inductively suppose we have picked $v_1,\ldots,v_{i-1}$ satisfying the necessary properties. To choose $v_i$ we must choose a neighbor of $v_{i'}$ such that 
    \begin{enumerate}
        \item $v_i\notin \{v_0,v_1,\ldots,v_{i-1}\}$ and
        \item $v_{i'}v_i$ does not have the same color as any of the previously selected $i-1$ edges and
        \item if $i'\ne 0$ then $v_i$ can not be adjacent to $v_0$.
    \end{enumerate}
    Note that the total number of bad choices for $v_i$ is at most $(i-1)+(i-1)+r-1$, with this last bound using that $v_{i'}$ has at most $r-1$ common neighbors with $v_0$ due to $G$ being $K_{2,r}$-free. In total then the number of bad choices is at most $2k+r-3<d$.  Thus we can choose $v_i$ satisfying the criteria at each stage, giving the result.
\end{proof}

\begin{proof}[Proof of Theorem \ref{thm:ex_gen} Part \ref{item:ex_gen_leaves}]
    If $e(G) > \frac{k-1}{2}V(G)$, by Lemma \ref{lem:min_degree_subgraph} we can find a subgraph $G'$ of $G$ with average degree at least $k$ and minimum degree at least $k/2$. Take a vertex $v_0\in V(G)$ with degree at least $k$. Let $T'$ be $T$ after removing all the leaves from $x_0$, noting that $e(T')=k-\ell$ has
    \[2e(T')+r-3 < k/2 \]
    and therefore we can apply Lemma \ref{lem:k2rfree_embedding} to find a rainbow $T'$ with $x_0$ sent to $v_0$. Suppose that $x_0$ has $m$ non-leaf neighbors. We greedily embed the $\ell$ leaves of $x_0$ to neighbors of $v_0$ with edge colors that aren't used elsewhere in the tree. This is possible since we have $\ell$ vertices to embed, the (unused) degree of $v_0$ is at least $k-m$, we must avoid the color and coordinate of the $k-m-\ell$ nonadjacent edges, and
    \[ (k-m)-2(k-m-\ell) = \ell.\]
\end{proof}
\section{Acknowledgments}
We are grateful to Dheer Noal Desai, Puck Rombach, and Jeremy Quail for insightful discussions and the latter for suggesting the problem. This work was started at the 2023 Graduate Research Workshop in Combinatorics, which was supported in part by NSF grant 1953985 and a generous award from the Combinatorics Foundation.

\bibliographystyle{plain}
\bibliography{references.bib}

\end{document}